\documentclass[12pt, reqno]{amsart}

\usepackage{fullpage}
\usepackage{setspace}
\onehalfspace

\usepackage{amsthm}
\theoremstyle{plain}
\newtheorem{X}{X}[section]
\newtheorem{Thm}[X]{Theorem}
\newtheorem{Prop}[X]{Proposition}

\numberwithin{equation}{section}

\usepackage{hyperref}
\hypersetup{
    pdftoolbar=true,                        
    pdfmenubar=false,                        
    pdffitwindow=false,                      
    pdfstartview={FitH},                    
    pdftitle={Products of shifted primes simultaneously taking perfect power values},    
    pdfauthor={Tristan Freiberg},    
    pdfcreator={Tristan Freiberg},
    pdfproducer={Tristan Freiberg},
    pdfnewwindow=true,                       
    colorlinks=true,
    linkcolor={black},                      
    citecolor={black},
    filecolor={black},
    urlcolor={black},         
}

\usepackage{amssymb}

\renewcommand{\le}{\ensuremath{\leqslant}}
\renewcommand{\ge}{\ensuremath{\geqslant}}

\newcommand{\br}[1]{\ensuremath{\left(#1\right)}} 
\newcommand{\abs}[1]{\ensuremath{\left\lvert#1\right\rvert}} 

\newcommand{\sums}[2]{\ensuremath{\sum_{\substack{#1 \\ #2}}}}

\subjclass{Primary 11N25, Secondary 11A25}

\title{Products of shifted primes simultaneously taking perfect power values}
\author{Tristan Freiberg}
\address{D\'epartement de math\'ematiques et de statistique \\
Universit\'e de Montr\'eal \\
Montr\'eal, QC \\
Canada}
\email{freiberg@dms.umontreal.ca}

\begin{document}

\begin{abstract}
Let $r \ge 2$ be an integer and let $A$ be a finite, nonempty set of nonzero integers. We will obtain a lower bound for the number of squarefree integers $n$, up to $x$, for which the products $\prod_{p \mid n} (p+a)$ (over primes $p$) are perfect $r$th powers for all $a \in A$. Also, in the cases $A = \{-1\}$ and $A = \{+1\}$, we will obtain a lower bound for the number of such $n$ with exactly $r$ distinct prime factors.
\end{abstract}

\maketitle

\section{Introduction}\label{Section 1}

If we pick a large integer close to $x$ at random, the probability that it is a perfect $r$th  power is around $x^{1/r}/x$. We might expect the shifted primes $p+a$ to behave more or less like random integers in terms of their multiplicative properties. Thus, if we take a large squarefree integer $n$ close to $x$, we might naively expect that $\sigma(n) = \prod_{p \mid n} (p+1) \approx n$ is an $r$th power with probability close to $x^{1/r}/x$. However, as we will see, the probability is much higher than this, indeed more than $x^{0.7038}/x$, for \emph{any} given $r$. We will even show that the likelihood of $\phi(n)$ and $\sigma(n)$ \emph{simultaneously} being (different) $r$th powers is more than $x^{0.2499}/x$. (As usual, $\phi$ denotes Euler's totient function and $\sigma$ denotes the sum-of-divisors function.) It would seem that $r$th powers are ``popular'' values for products of shifted primes in general.

If we only count those $n$ with exactly $r$ prime factors, we will show that the number of such $n$ up to $x$ for which $\phi(n)$ is a perfect $r$th power is $\gg x^{1/r}/(\log x)^{r+2}$, and likewise for $\sigma(n)$. Thus there are $\gg x^{1/2}/(\log x)^4$ integers $n \le x$ for which $n = pq$, $p$ and $q$ distinct primes, and $(p-1)(q-1)$ is a square. This may be seen as an ``approximation'' to the well-known conjecture that there are infinitely many primes $p$ for which $p-1$ is a square. It is easily seen that there is at most one prime $p$ for which $p+1$ is a perfect $r$th power ($r \ge 2$), namely $3+1 = 2^2$, $7+1=2^3$, and so on.

Given an integer $r \ge 2$ and a finite, nonempty set $A$ of nonzero integers, let
\[
 \mathcal{B}(x;A,r) = \left\{n \le x : \textrm{$n$ is squarefree and $\textstyle{\prod_{p \mid n}(p+a)}$ is an $r$th power for all $a \in A$}\right\}.
\]
Banks et.\ al.\ \cite{BFPS2004} proved, among several other results, that $\abs{\mathcal{B}(x;\{-1\},2)}, \abs{\mathcal{B}(x;\{+1\},2)} \ge x^{0.7039 - o(1)}$, and that $\abs{\mathcal{B}(x;\{-1,+1\},2)} \ge x^{1/4 - o(1)}$, where here and throughout, $o(1)$ denotes a function tending to $0$ as $x$ tends to infinity. The first theorem generalizes both of these results.

\begin{Thm}\label{Theorem 1.1}
 Fix an integer $r \ge 2$, and a finite, nonempty set $A$ of nonzero integers. As $x \to \infty$, we have
\begin{align}\label{(1.1)}
 \abs{\mathcal{B}(x;A,r)} \ge x^{1/2\abs{A} - o(1)}.
\end{align}
Moreover, if $\abs{A} = 1$, then as $x \to \infty$, we have
\begin{align}\label{(1.2)}
 \abs{\mathcal{B}(x;A,r)} \ge x^{0.7039 - o(1)}.
\end{align}
\end{Thm}

In the case $A = \{-1\}$ (respectively $A = \{+1\}$), $\mathcal{B}(x;A,r)$ is the set of squarefree integers $n$ up to $x$ for which $\phi(n)$ (respectively $\sigma(n)$) is an $r$th power. There is no condition on the number of prime factors of $n$, but the next theorem concerns
 \begin{align*}
\mathcal{B}^{*}(x;-1,r)  & = \{n \le x : \textrm{$n$ is squarefree, $\omega(n) = r$ and $\phi(n)$ is an $r$th power}\}, \\
\mathcal{B}^{*}(x;+1,r)  & = \{n \le x : \textrm{$n$ is squarefree, $\omega(n) = r$ and $\sigma(n)$ is an $r$th power}\},
\end{align*}
where $\omega(n)$ is the number of distinct prime factors of $n$.
\begin{Thm}\label{Theorem 1.2}
Fix an integer $r \ge 2$. For all sufficiently large $x$, we have
\begin{align}\label{(1.3)}
\abs{\mathcal{B}^{*}(x;-1,r)}, \,\abs{\mathcal{B}^{*}(x;+1,r)} \gg \frac{rx^{1/r}}{(\log x)^{r+2}}.
\end{align}
The implied constant is absolute.
\end{Thm}

The proof of Theorem \ref{Theorem 1.1} (Section \ref{Section 3}) is an extension of the proof by Banks et.\ al.\ \cite{BFPS2004} of the aforementioned special cases of Theorem \ref{Theorem 1.1}. It employs some of the ideas of Erd\H{o}s \cite{E1935, E1956} upon which Alford, Granville and Pomerance \cite{AGP1994} based their proof that there are infinitely many Carmichael numbers. The proof of Theorem \ref{Theorem 1.2} (Section \ref{Section 4}) introduces a new method, which, as we will explain, is an application of the ideas of Goldston, Pintz and Y{\i}ld{\i}r{\i}m \cite{GPY2009}. 

\section{Preliminaries}\label{Section 2}

Theorem \ref{Theorem 1.1} is a consequence of the first four results of this section, and we use the fifth in the proof of Theorem \ref{Theorem 1.2}.

An integer $n$ is called $y$-smooth if $p \le y$ for every prime $p$ dividing $n$. Given a polynomial $F(X) \in \mathbb{Z}[X]$ and numbers $x \ge y \ge 2$, let
\[
 \pi_F(x,y) = \abs{\{p \le x : \textrm{$F(p)$ is $y$-smooth}\}}.
\]
In the case $F = X - 1$, Erd\H{o}s \cite{E1935} proved that there exists a number $\epsilon \in (0,1)$ such that $\pi_F(x,x^{\epsilon}) \gg_{\epsilon} \pi(x)$ (where $\pi(x)$ is the number of primes up to $x$), for all large $x$ depending on the choice of $\epsilon$. Several authors have improved upon this, the next two results being the best so far obtained.
\begin{Thm}\label{Theorem 2.1}
 Fix a nonzero integer $a$ and let $F(X) = X+a$. For some absolute constant $c$, we have
\[
 \pi_F(x,y) > \frac{x}{(\log x)^c}
\]
for all sufficiently large $x$, provided $y \ge x^{0.2961}$. 
\end{Thm}
\begin{proof}
 See \cite[Theorem 1]{BH1998}.
\end{proof}

\begin{Thm}\label{Theorem 2.2}
 Let $F$ be a polynomial with integer coefficients. Let $g$ be the largest of the degrees of $F$ and let $k$ be the number of distinct irreducible factors of $F$ of degree $g$. Suppose that $F(0) \ne 0$ if $g = k = 1$, and let $\epsilon$ be any positive real number. Then the estimate
\[
 \pi_F(x,y) \asymp \frac{x}{\log x}
\]
holds for all sufficiently large $x$, provided $y \ge x^{g+\epsilon - 1/2k}$.
\end{Thm}
\begin{proof}
 See \cite[Theorem 1.2]{DMT2001}.
\end{proof}

For a finite additive abelian group $G$, denote by $n(G)$ the length of the longest sequence of (not necessarily distinct) elements of $G$, no nonempty subsequence of which sums to $0$, the additive identity of $G$. For instance, if $G = (\mathbb{Z}/2\mathbb{Z})^m$, then $n(G) \le m$, for any sequence of $m+1$ elements of $G$ contains a nonempty subsequence whose elements sum to $(0,\ldots,0) \bmod 2$, as can be seen by considering that such a sequence contains $2^{m+1} - 1 > 2^m = \abs{G}$ nonempty subsequences. For any group $G$ of order $m$, then any sequence of $m$ elements contains a nonempty subsequence whose sum is $0$, hence $n(G) \le m-1$. The next theorem, due to van Emde Boas and Kruyswijk \cite{EK1969}, gives a nontrivial upper bound for $n(G)$.
 
\begin{Thm}\label{Theorem 2.3}
 If $G$ is a finite abelian group and $m$ is the maximal order of an element in $G$, then $n(G) < m(1 + \log (\abs{G}/m))$.
\end{Thm}
\begin{proof}
 See \cite{EK1969}. A proof is also given in \cite[Theorem 1.1]{AGP1994}.
\end{proof}

The following proposition shows that there may be many sequences in $G$ whose elements sum to $0$. 
\begin{Prop}\label{Proposition 2.4}
 Let $G$ be a finite abelian group and let $r > k > n = n(G)$ be integers. Then any subsequence of $r$ elements of $G$ contains at least $\binom{r}{k} \big\slash \binom{r}{n}$ distinct subsequences of length at most $k$ and at least $k - n$, whose sum is the identity. 
\end{Prop}
\begin{proof}
 See \cite[Proposition 1.2]{AGP1994}.
\end{proof}

We will use the well-known Siegel-Walfisz theorem in the proof of Theorem \ref{Theorem 1.2}.
\begin{Thm}[Siegel-Walfisz]\label{Theorem 2.5}
For any positive number $B$, there is a constant $C_B$ depending only on $B$, such that 
 \[
  \sum_{\substack{p \le N \\ p \equiv a \bmod q}} \log p = \frac{N}{\phi(q)} + O\br{N\exp\br{-C_B(\log N)^{1/2}}}
 \]
whenever $(a,q) = 1$ and $q \le (\log N)^B$.
\end{Thm}
\begin{proof}
 See \cite[Chapter 22]{D2000}.
\end{proof}

\section{Proof of Theorem 1.1}\label{Section 3}

The following proof hinges on Theorem \ref{Theorem 2.3} and Proposition \ref{Proposition 2.4}, which are key ingredients in the celebrated proof of Alford, Granville and Pomerance \cite{AGP1994} that there are infinitely many Carmichael numbers. (A Carmichael number is a composite number $n$ for which $a^n \equiv a \bmod n$ for all integers $a$.) In fact it is shown in \cite[Theorem 1]{AGP1994} that the number of Carmichael numbers $C(x)$ up to $x$ satisfies $C(x) \ge x^{\beta - \epsilon}$ for any $\epsilon > 0$ and all large $x$ depending on the choice of $\epsilon$, where
\[
 \beta = \frac{5}{12}\br{1 - \frac{1}{2\sqrt{e}}} = 0.29036\ldots .
\]

Using a variant of the construction in \cite{AGP1994}, Harman \cite{H2005} proved that $\beta = 0.3322408$ is admissible, and combining the ideas of \cite{AGP1994, BFPS2004, H2005}, Banks \cite{B2009} established the following result.
\begin{Thm}[{\cite[Theorem 1]{B2009}}]\label{Theorem 3.1}
For every fixed $C < 1$, there is a number $x_0(C)$ such that for all $x \ge x_0(C)$ the inequality
\[
 \abs{\{n \le x : \textrm{$n$ is Carmichael and $\phi(n)$ is an $r$th power}\}} \ge x^{\beta - \epsilon}
\]
holds, with $\beta = 0.3322408$ and any $\epsilon > 0$, for all positive integers $r \le \exp\br{(\log\log x)^C}$.
\end{Thm}
(Harman \cite{H2008} has subsequently proved that $\beta = 0.7039\times 0.4736 > 1/3$ is admissible here.) The method of the proof may yield further interesting results. 

Theorems \ref{Theorem 2.1} and \ref{Theorem 2.2} are also crucial, and it will be manifest that extending the admissible range for $y$ in those theorems will lead to better estimates for $\abs{\mathcal{B}(x;A,r)}$. Explicitly, if $F(X) = {\textstyle\prod_{a \in A}}(X+a)$ and
\[
 \pi_F(x,x^{\epsilon}) \asymp_{F,\epsilon} \frac{x}{\log x}
\]
holds, then the following proof yields $\abs{\mathcal{B}(x;A,r)} \ge x^{1 - \epsilon - o(1)}$. It is suspected that any positive $\epsilon$ is admissible, in which case we would have $\abs{\mathcal{B}(x;A,r)} = x^{1 - o(1)}$.

\begin{proof}[Proof of Theorem \ref{Theorem 1.1}]
Fix an integer $r \ge 2$ and a set $A = \{a_1,\ldots,a_s\}$ of nonzero integers. Let $x$ be a large number, and let
\begin{align}\label{(3.1)}
y = \frac{\log x}{\log\log x}.
\end{align}
Let $t = \pi(y)$, and let $G = (\mathbb{Z}/r\mathbb{Z})^{st}$, so that by Theorem \ref{Theorem 2.5},
\begin{align}\label{(3.2)}
n(G) < r(1 + \log\abs{G}/r) = r(1 + (st-1)\log r).
\end{align} 
Fix any number $\epsilon \in (0,1/3s)$, and let
\begin{align*}
u =
\begin{cases}
(0.2961)^{-1} & \textrm{if $s = 1$,} \\
\br{1 + \epsilon - \frac{1}{2s}}^{-1} & \textrm{if $s \ge 2$.}
\end{cases}
\end{align*}
Let
\[
F(X) = (X+a_1)(X+a_2)\cdots (X+a_s),
\]
and let
\[
S_F(y^u,y) = \{p \le y^u : \textrm{$F(p)$ is $y$-smooth} \} 
 = \{p \le y^u : \textrm{$p + a_1,\ldots,p+a_s$ are $y$-smooth} \}.
\]
We may suppose $x$, and hence $y$, is large enough so that, by Theorem \ref{Theorem 2.1} and Theorem \ref{Theorem 2.2}, 
\begin{align}\label{(3.3)}
\abs{S_F(y^u,y)} = \pi_F(y^u,y) \gg \frac{y^u}{(\log y^u)^c}
\end{align}
for some constant $c$. (We may suppose $c = 1$ in the case $s \ge 2$.) Finally, let
\begin{align}\label{(3.4)}
k = \left[\frac{\log x}{\log y^u}\right],
\end{align}
where $[\alpha]$ denotes the integer part of a real number $\alpha$.

By \eqref{(3.1)}, \eqref{(3.3)} and \eqref{(3.4)},
\[
\frac{\pi_F(y^u,y)}{k} \gg \frac{(\log x)^{u-1}}{(\log\log x)^{u-1+c}},
\]
and by \eqref{(3.1)}, \eqref{(3.2)} and \eqref{(3.4)}, 
\begin{align}\label{(3.5)}
\frac{k}{n(G)} \gg_{r,s} \frac{\log x/\log y^u}{t} \gg  \log\log x,
\end{align}
because $t = \pi(y) \sim y/\log y$ as $y \to \infty$, by the prime number theorem. Therefore, since $u > 1$, we may assume $x$ is large enough so that
\begin{align}\label{(3.6)}
n(G) < k < \pi_F(y^u,y).
\end{align}

For primes $p \in S_F(y^u,y)$ and integers $a \in A$, we may write
\[
p+a = 2^{\beta^{(a)}_1}3^{\beta^{(a)}_2}\cdots p_t^{\beta^{(a)}_t},
\]
where $\beta^{(a)}_i$, $1 \le i \le t$, are nonnegative integers. We define
\[
\mathbf{v}_p =(\beta^{(a_1)}_1,\ldots,\beta^{(a_1)}_t,\beta^{(a_2)}_1,\ldots,\beta^{(a_2)}_t,\ldots,\beta^{(a_s)}_1,\ldots,\beta^{(a_s)}_t)
\]
as the ``exponent vector'' for $p$. For a subset $R$ of $S_F(y^u,y)$, $\textstyle{\prod_{p \in R}} (p + a)$
is an $r$th power for every $a \in A$ if and only if 
\[
\sum_{p \in R} \mathbf{v}_p \equiv \mathbf{0} \bmod r,
\]
where $\mathbf{0} \bmod r$ is the zero element of $G$. If, moreover, $R$ is of size at most $k$, then by \eqref{(3.4)},
\[
\prod_{p \in R} p \le y^{uk} \le x.
\]
Thus
\begin{align}\label{(3.7)}
 \abs{\mathcal{B}(x;A,r)} \ge \abs{\left\{R \subseteq S_F(y^u,y) : \textrm{$\abs{R} \le k$ and $\textstyle{\sum_{p \in R}} \mathbf{v}_p \equiv \mathbf{0} \bmod r$}\right\}},
\end{align}
as distinct subsets $R \subseteq S_F(y^u,y)$ give rise to distinct integers $n$, by uniqueness of factorization.

Because of \eqref{(3.6)}, we may deduce from Proposition \ref{Proposition 2.4} that the right-hand side of \eqref{(3.7)} is at least
\[
\binom{\pi_F(y^u,y)}{k} \Big\slash \binom{\pi_F(y^u,y)}{n(G)} \ge \br{\frac{\pi_F(y^u,y)}{k}}^{k}\pi_F(y^u,y)^{-n(G)} := x^{f(x)},
\]
where
\[
f(x) = (k - n(G))\frac{\log \pi_F(y^u,y)}{\log x} - \frac{k\log k}{\log x}.
\]
Letting $x$ tend to infinity and using \eqref{(3.1)}, \eqref{(3.3)}, \eqref{(3.4)}, and \eqref{(3.5)}, we see that $f(x) = 1 - 1/u - o(1)$. Therefore, as $x \to \infty$, we have
\[
 \abs{\mathcal{B}(x;A,r)} \ge x^{1 - 1/u - o(1)},
\]
and Theorem \ref{Theorem 1.1} follows by our choice for $u$, and letting $\epsilon$ tend to $0$ in the case $s \ge 2$. 
\end{proof}

\section{Proof of Theorem 1.2}\label{Section 4}

We use a different approach to prove Theorem \ref{Theorem 1.2}. The proof is ``inspired'' by the breakthrough results of Goldston, Pintz and Y{\i}ld{\i}r{\i}m \cite{GPY2009} on short intervals containing primes. Basically, their proof begins with the observation that if $W(n)$ is a nonnegative weight and
\begin{align}\label{(4.1)}
\sum_{N < n \le 2N}\br{\sum_{h \le H} \vartheta(n+h) - \log (2N+H)}W(n)
\end{align}
is positive, then for some $n \in (N,2N]$, the interval $(n,n+H]$ contains at least $2$ primes. Here and in the sequel, 
\begin{align*}
\vartheta(n) 
=
\begin{cases}
\log n & \textrm{if $n$ is prime,} \\
0 & \textrm{otherwise.}
\end{cases}
\end{align*}
Goldston, Pintz and Y{\i}ld{\i}r{\i}m were able to obtain a nonnegative weight $W(n)$ for which \eqref{(4.1)}, with $H = \epsilon\log N$, is positive for all sufficiently large $N$. In our problem, we will be led to consider
\[
 \sum_{n \le N} \br{\sum_{a \le H} \vartheta(a^rn+1) - (r-1)\log (H^rN+1)}
\]
(see \eqref{(4.3)}). A lower bound for this expression corresponds to a lower bound for the number of $n \le N$ for which $\{a^rn+1 : a \le H\}$ contains at least $r$ primes. As we do not require $H$ to be ``short'' compared to $N$, we may take $H = r\log N$: then the weight $W(n) = 1$ works, and the problem is much easier.

\begin{proof}[Proof of Theorem \ref{Theorem 1.2}]
Throughout the proof, $r \ge 2$ is a fixed integer, and $n,a,a_1,a_2,\ldots$ are positive integers. Observe that if, for some $n$,
\[
\ell_i = a_i^rn+1, \quad i = 1,\ldots,r
\]
are distinct primes, then
\[
\phi(\ell_1\cdots \ell_r) = (a_1\cdots a_rn)^r.
\]
If the primes $\ell_i$ are of the form $a_i^rn-1$ then $\sigma(\ell_1\cdots \ell_r) = (a_1\cdots a_rn)^r$. We will prove that \eqref{(1.3)} holds for $\abs{\mathcal{B}(x;-1,r)}$, provided $x$ is sufficiently large, and the same proof applies to $\abs{\mathcal{B}(x;+1,r)}$ if we consider primes of the form $a_i^rn-1$ rather than $a_i^rn+1$. 

Let $N$ be a parameter tending monotonically to infinity and set $H = r\log N$. Let $\mathcal{A}(N)$ be the set of $n \le N$ for which 
\[
\mathcal{C}_n = \{a^rn+1 : a \le H\} \cap \mathcal{P}
\]
(where $\mathcal{P}$ is the set of all primes) contains at least $r$ primes. We will show that
\begin{align}\label{(4.2)}
\abs{\mathcal{A}(N)} \gg \frac{N}{\log N},
\end{align}
but first we will describe how this implies a lower bound for $\abs{\mathcal{B}(x;-1,r)}$.

Every $n \in \mathcal{A}(N)$ gives rise, via $\mathcal{C}_n$, to some $\ell_1\cdots \ell_r \in \mathcal{B}((H^rN+1)^r;-1,r)$, though different $n$ may give rise to the same $r$-tuple of primes. On the other hand, given $n \in \mathcal{A}(N)$ and a prime $p = a^rn+1 \in \mathcal{C}_n$, each $m \in \mathcal{A}(N)$ for which $\mathcal{C}_m = \mathcal{C}_n$ corresponds to a solution to $a^rn = b^rm$, $b \le H$. Therefore there can be at most $H$ different $n \in \mathcal{A}(N)$ giving rise to the same element of $\mathcal{B}((H^rN+1)^r;-1,r)$. Consequently,
\[
\abs{\mathcal{B}((H^rN+1)^r;-1,r)} \ge \frac{\abs{\mathcal{A}(N)}}{H} \gg \frac{N}{r(\log N)^2}
\]
by \eqref{(4.2)}, and \eqref{(1.3)} follows.

We will now establish \eqref{(4.2)}. We will show that for all large $N$,
\begin{align}\label{(4.3)}
S(N) = \sum_{n \le N} \br{\sum_{a \le H} \vartheta(a^rn+1) - (r-1)\log (H^rN+1)} \gg rN\log N.
\end{align}
Consequently $\mathcal{A}(N)$ is nonempty for large $N$. Indeed, if \eqref{(4.3)} holds then
\begin{align*}
rN\log N  \ll S(N) & \le 
\sum_{n \in \mathcal{A}(N)} \br{\sum_{a \le H} \vartheta(a^rn+1) - (r-1)\log (H^rN+1)} \\ & \le
\abs{\mathcal{A}(N)} H \log (H^rN+1),
\end{align*}
and \eqref{(4.2)} follows because $\log (H^rN+1) \sim \log N$.

For the evaluation of $S(N)$, first note that
\[
\sum_{n \le N}  \sum_{a \le H} \vartheta(a^rn+1) = \sum_{a \le H} \sums{p \le a^rN+1}{p \equiv 1 \bmod a^r}  \log p.
\]
Since $a^r \ll_r (\log N)^r$ for $a \le H$, we may apply Theorem \ref{Theorem 2.5} to the last sum. We have
\[
\sums{p \le a^rN+1}{p \equiv 1 \bmod a^r}  \log p = \frac{a^rN}{\phi(a^r)} + O\br{\frac{a^rN}{\phi(a^r)(\log N)^2}} \sim \frac{a}{\phi(a)}N.
\]
Therefore, by the well-known estimate
\[
 \sum_{a \le H} \frac{a}{\phi(a)} \sim cH, \quad c = \prod_p \br{1 + \frac{1}{p(p-1)}} = 1.943596\ldots,
\]
we have
\[
\sum_{n \le N}  \sum_{a \le H} \vartheta(a^rn+1) \sim 
N\sum_{a \le H} \frac{a}{\phi(a)} \sim cNH.
\]
Also,
\[
\sum_{n \le N} (r-1) \log (H^rN+1) \sim N(r-1)\log N,
\]
so combining all of this yields
\[
S(N) \sim N(cH - (r-1)\log N) \gg rN\log N,
\]
hence \eqref{(4.3)}.
\end{proof}

\section{Acknowledgements}\label{Section 5}
I would like to thank Andrew Granville for introducing me to the problems considered in this note, and for his help in the preparation of it. I would also like to thank Christian Elsholtz for some helpful comments.

\bibliographystyle{article}

\end{document}